\documentclass[a4paper,10pt]{amsart}

\usepackage{amsmath,amssymb,amsthm,a4wide}

\usepackage{tikz}
\usetikzlibrary{shapes}
\usetikzlibrary{intersections}

\usepackage{graphicx}
\usepackage{graphics}

\newtheorem{theorem}{Theorem}
\newtheorem*{thm}{Theorem}
\newtheorem*{lemma}{Lemma}

\begin{document}

\title[Improved Bounds for Random TSP]{New Bounds for the Traveling Salesman Constant}
\author{Stefan Steinerberger}
\address{Mathematisches Institut, Universit\"at Bonn, Endenicher Allee 60, 53115 Bonn, Germany}
\email{steinerb@math.uni-bonn.de}

\begin{abstract} 
Let $X_1, X_2, \dots, X_n$  be independent and uniformly distributed random variables in the
unit square $[0,1]^2$ and let $L(X_1, \dots, X_n)$
be the length of the shortest traveling salesman path through these points. In 1959, Beardwood, Halton \& Hammersley proved the existence of a universal constant $\beta$ such that
$$ \lim_{n \rightarrow \infty}{n^{-1/2}L(X_1, \dots, X_n)} = \beta \qquad \mbox{almost surely.}$$
The best bounds for $\beta$ are still the ones originally established by Beardwood, Halton \& Hammersley $0.625 \leq \beta \leq 0.922$.  We slightly improve both upper and lower bounds.
\end{abstract}
\maketitle

\section{Introduction and statement of results}
For given points $x_1, x_2, \dots, x_n \subset [0,1]^2$, let $L(x_1, \dots, x_n)$ denote 
the length of the shortest traveling salesman path through all this points. It was realized
early (e.g. Fejes \cite{fejes} in 1940, Verblunsky \cite{verb} in 1951 and
Few \cite{few} in 1955) that there are uniform estimates
$$ L(x_1, \dots, x_n) \leq c_1\sqrt{n}+c_2$$
for some constants $c_1, c_2$.
If the points are chosen at random, one would expect a universality phenomenon: finding
the optimal path is in some sense 'equivalent' to finding the optimal path through the points in
many small subset of the unit square and then patching these together: the problem is
self-similar on a smaller scale and this should imply an averaging effect. That this is indeed the case  constitutes one of the first limit theorems in combinatorial optimization \cite{beard}.

\begin{thm}[Beardwood, Halton \& Hammersley, 1959] Let $X_1, X_2, \dots, X_n, \dots$  be i.i.d. uniformly distributed random variables in $[0,1]^2$. There exists a universal constant $\beta$ such that
$$ \lim_{n \rightarrow \infty}{\frac{L(X_1, \dots, X_n)}{\sqrt{n}}} = \beta$$
with probability 1.
\end{thm}
The statement is by now classic and very well-known (see, for example, the textbooks of
Applegate, Bixby, Chvatal \& Cook \cite{apple}, Finch \cite{finch}, Gutin \& Punnen \cite{gut}, Steele \cite{steeletext}
or Venkatesh \cite{venk} or even a popular-science book \cite{cook}). It is relatively easy to deduce
that if the points $X_i$ are random following an absolutely continuous probability distribution $f(x)$ on $\mathbb{R}^2$, then
$$ \lim_{n \rightarrow \infty}{\frac{L(X_1, \dots, X_n)}{\sqrt{n}}} = \beta\int_{[0,1]^2}{f(x)^{
\frac{1}{2}}dx}.$$
The Beardwood-Halton-Hammersley limit law is true for various other problems 
(e.g. minimal spanning tree, Steiner trees, ...) with a constant depending on the functional: a unified approach to the theory is given by Steele's limit theorem \cite{steele}. Interestingly and despite considerable effort, the constant is not known
in any of the aforementioned cases. In case of the traveling salesman, Beardwood, Halton \& Hammersley themselves proved that
$$ 0.625 = \frac{5}{8} \leq \beta \leq \beta_{BHH} \sim 0.92116\dots,$$
where 
$$\beta_{BHH} = 2\int_{0}^{\infty}{\int_{0}^{\sqrt{3}}{\sqrt{z_1^2+z_2^2}\exp{\left(-\sqrt{3}z_1\right)}
\left(1-\frac{z_2}{\sqrt{3}}\right)dz_2} dz_1}.$$
It should be noted that Beardwood, Halton \& Hammersley actually claim to prove the better result $\beta \leq 0.92037 \dots$ (a
statement reiterated in many different books and papers), however, their computation relies on numerical
integration and we believe this to be the origin of the error: for the convenience of the
reader, we have quickly surveyed their argument (and the integral to be evaluated) below. Despite the relative fame of the Beardwood-Halton-Hammersley theorem, there has been no improvement in the constant over
the years; a series of papers \cite{johnson, percus, val} carrying out numerical estimates with large data sets suggest $\beta \sim 0.712$.
The purpose of this paper is to draw some attention to the problem, describe the
existing original arguments and to improve them.
\begin{theorem} We have
$$  \frac{5}{8} + \frac{19}{5184} \leq \beta  \leq \beta_{BHH} - \varepsilon_0$$
for some explicit 
$$\varepsilon_0 > \frac{9}{16}10^{-6}.$$
\end{theorem}
We have an explicit representation of $\varepsilon_0$ as an integral in $\mathbb{R}^7$: a concentration of measure effect turns Monte-Carlo estimates into a highly stable method and suggests that actually
$$\varepsilon_0 \sim 0.0148\dots,$$
however, we consider the underlying idea to be of greater interest than the actual numerical improvement -- in addition,
certain natural generalizations of our method should be able to give at least $\beta \leq 0.891$ if one assumes that certain integrals in high dimensions can be evaluated (details are given below). While additional improvements of the upper bound may lead to integrals
whose evaluations become nontrivial, the approach is conceptually clear: further improving the lower bound, however, seems more challenging and in need of new ideas. As of this moment, we know of no methodical approach how this could be accomplished.

\section{Proof of the Upper Bound}

\subsection{Reduction to Poisson processes.} The core of the proof is in the stochastic treatment of $n$ random
points in $[0,1]^2$ locally on the scale $n^{-1/2}$. At this scale the law of small numbers (see e.g. \cite{sv}) implies that
the process behaves essentially like a Poisson process with intensity $n$. This property was exploited by Beardwood, Halton \& Hammersley; using their result, we can replace the $n$ random points with a Poisson process with intensity $n$, which simplifies further computations (this argument was pointed out to me by J. Michael Steele). 

\begin{lemma} Let $\mathcal{P}_n$ denote a Poisson process with intensity $n$ on $[0,1]^2$. Then
$$ \lim_{n \rightarrow \infty}{\frac{\mathbb{E} L(\mathcal{P}_{n})}{\sqrt{n}}} = \beta.$$
\end{lemma}
The idea is rather simple: the number of points in a Poisson process (i.e. the Poisson
distribution) has mean $n$ and variance $n$. This means that we usually expect $\left|\#\mathcal{P}_{n} - n\right|
\sim \sqrt{n}$, which is rather small compared to $n$. The expected length of a traveling
salesman path lies somewhere between $\sim \beta\sqrt{n-\sqrt{n}}$ and $\sim \beta\sqrt{n+\sqrt{n}}$, the difference
of which is $\sim 1$ and thus of smaller order -- if we now assume the Beardwood, Halton \& Hammersley result,
this implies convergence for all cases concentrated here. Additionally, to deal with the other cases, it suffices to show that is very unlikely to have an unusually large amount of points and that the uniform bound of Few suffices. We leave the details to the interested reader and remark that the inverse statement (i.e. that the result
for the Poisson process implies the desired result is actually due to Beardwood, Halton \& Hammersley).

\subsection{The original argument.}
This section describes the original argument due to Beardwood-Halton-Hammersley. It should
be noted that a very similar argument was already used few years earlier by Few \cite{few}. Let $X$
be a Poisson process with intensity $n$ in the unit square $[0,1]^2$. We look at the set
$$ X^* = \left\{x \in X: \pi_2(x) \leq \frac{\sqrt{3}}{\sqrt{n}}\right\}, $$
where $\pi_2$ is the projection onto the second component. 
Instead of asking for a traveling
salesman tour through all the points, we merely ask for one through this particular strip. The
entire unit square is cut into stripes and within each strip a local path gets constructed: in the
end they all get connected to yield a fully valid traveling salesman path.
 The simplest solution locally within a strip is to order the points in $X^*$ with respect to the first coordinate, i.e. order
them in such a way that
$$ \pi_1(x_1) < \pi_1(x_2) < \pi_1(x_3) < \dots $$
and then simply connected the points in that order.\\

\begin{figure}[h!]
\begin{tikzpicture}[xscale = 1, yscale =1]
\draw [thick] (0,0) --(10,0);
\draw [thick](0,1.5) --(10,1.5);
\draw [thick] (0,0) --(0,1.5);
\draw [thick] (10,0) --(10,1.5);
\draw [fill] (1,0.5) circle [radius=0.05];
\draw  (1,0.5) --(1.3,1.15);
\draw [fill] (1.3,1.15) circle [radius=0.05];
\draw  (2.5,0.15) --(1.3,1.15);
\draw [fill] (2.5,0.15) circle [radius=0.05];
\draw  (2.5,0.15) --(3,1.25);
\draw [fill] (3,1.25) circle [radius=0.05];
\draw  (3,1.25) --(3.3,0.75);
\draw [fill] (3.3,0.75) circle [radius=0.05];
\draw  (4.8,1.35) --(3.3,0.75);
\draw [fill] (4.8,1.35) circle [radius=0.05];
\draw  (4.8,1.35) --(5.5,1);
\draw [fill] (5.5,1) circle [radius=0.05];
\draw  (6.3,0.75) --(5.5,1);
\draw [fill] (6.3,0.75) circle [radius=0.05];
\draw  (6.3,0.75) --(7,0.15);
\draw [fill] (7,0.15) circle [radius=0.05];
\draw  (8,1) --(7,0.15);
\draw [fill] (8,1) circle [radius=0.05];
\end{tikzpicture}
\caption{A strip containing some points.}
\end{figure}
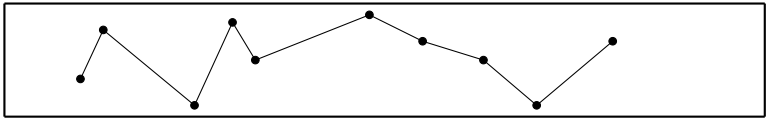

\begin{thm}[Beardwood, Halton \& Hammersley, 1959] Let $X$  be a Poisson process with intensity $n$ in $[0,1]^2$ and let $F$ be the length of the path constructed in the way described above. Then
$$ \lim_{n \rightarrow \infty}{\frac{\mathbb{E}F(X)}{\sqrt{n}}} = 0.92116\dots$$
\end{thm}

\begin{proof}[Sketch of the proof.] We restrict the Poisson process with intensity $n$ to
the strip $\pi_2(x) \leq \sqrt{3}/\sqrt{n}.$
Then the real random variables
$$ \left\{\pi_1(x) : \pi_2(x) \leq \frac{\sqrt{3}}{\sqrt{n}}\right\}$$
are distributed following a Poisson process with intensity $\sqrt{3n}$ on $[0,1]$. Ordering the points with
respect to increasing first coordinate will give $x-$coordinates whose consecutive differences are exponentially distributed 
$$|\pi_1(x_{i+1}) - \pi_1(x_i)| \sim \sqrt{3n}\exp{\left(-\sqrt{3n}z\right)},$$
while the $y-$coordinates are i.i.d. distributed following the uniform distribution on $[0, \sqrt{3}/\sqrt{n}]$.
Therefore, the expected distance in joining one point to the next is given by
$$  \mathbb{E}\|x_i - x_{i+1}\| =
2n\int_{0}^{\infty}{\int_{0}^{\sqrt{3/n}}{\sqrt{z_1^2+z_2^2}\exp{\left(-\sqrt{3n}z_1\right)}
\left(1-\frac{\sqrt{n}z_2}{\sqrt{3}}\right)dz_2} dz_1}.$$
Substitution allows to rewrite the integral as
$$  \mathbb{E}\|x_i - x_{i+1}\| = \left(2\int_{0}^{\infty}{\int_{0}^{\sqrt{3}}{\sqrt{z_1^2+z_2^2}\exp{\left(-\sqrt{3}z_1\right)}
\left(1-\frac{z_2}{\sqrt{3}}\right)dz_2} dz_1}\right)\frac{1}{\sqrt{n}}.$$
Since we are actually joining all $n$ points, we have to jump from one strip to another 
$\sim \sqrt{n/3}$ times and each time the jump is of order $\sim 1/\sqrt{n}$; this implies that
the contribution coming from these jumps is of order $\mathcal{O}(1)$ and the total expected length is simply given by $n$ times the expected
length of a single jump, which gives
$$  \left(2\int_{0}^{\infty}{\int_{0}^{\sqrt{3}}{\sqrt{z_1^2+z_2^2}\exp{\left(-\sqrt{3}z_1\right)}
\left(1-\frac{z_2}{\sqrt{3}}\right)dz_2} dz_1}\right)\sqrt{n} \sim 0.92116\sqrt{n}.$$
\end{proof}
The underlying 'layer'-method is easily extended to higher dimensions and variable densities, see a paper of Borovkov \cite{bor}.

\subsection{Changing variables.} The argument contains all the necessary ingredients for our improved local construction: following the steps
outlined above, we will study Poisson processes with intensity $n$ in the strip
$$ \left\{(x,y) \in \mathbb{R}^2: 0 \leq x \leq 1 \wedge 0 \leq y \leq \frac{\sqrt{3}}{\sqrt{n}}\right\},$$
which, following the same variable transformation as above, turns into studying local properties of the Poisson process with intensity 1 in the infinite strip $ \left\{(x,y) \in \mathbb{R}^2: 0 \leq y \leq \sqrt{3}\right\}.$ We construct the Poisson distribution indirectly in the following way: since we are interested in the lengths of paths through a local number of points and the strip has a translation symmetry, we may assume the first point to be given by $p_1 = (0, y_1)$, where $y_1$ is uniformly distributed on $[0, \sqrt{3}]$. Adding now iteratively exponentially distributed random variables with parameter $\sqrt{3}$ to the first variables and replacing the second component by independent uniformly distributed random variables in $[0, \sqrt{3}]$ yields the Poisson process with intensity $1$ in the strip.

\subsection{Counting zigzags}
The key observation in our improvement is the following: the Beardwood-Halton-Hammersley method is locally quite
a bad if we encounter what we will informally call a zig-zag structure in the points: 4 consecutive points
with a small difference in the $x-$coordinate but a large difference in the $y-$coordinate. More precisely,
we will say that 4 points $p_1, p_2, p_3, p_4$ (ordered such that their $x-$coordinates increase) form
a zigzag if
$$ \| p_1 - p_3\| + \|p_3 - p_2\| + \|p_2 - p_4\|  \leq \| p_1 - p_2\| + \| p_2 - p_3\| + \| p_3 - p_4\|.$$
Given a zigzag,  it is advantageous to locally change the structure of the path. 
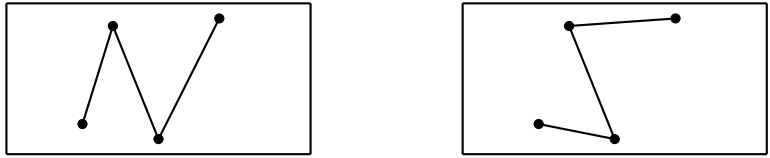
\begin{figure}[h!]
\begin{tikzpicture}[xscale = 2, yscale = 2]
\draw [thick] (0,0) --(2,0);
\draw [thick](0,1) --(2,1);
\draw [thick] (0,0) --(0,1);
\draw [thick] (2,0) --(2,1);
\draw [fill] (0.5,0.2) circle [radius=0.03];
\draw [fill] (0.7,0.85) circle [radius=0.03];
\draw [fill] (1,0.1) circle [radius=0.03];
\draw [fill] (1.4,0.9) circle [radius=0.03];
\draw [thick] (0.5,0.2) --(0.7,0.85) --(1,0.1) --(1.4,0.9);

\draw [thick] (3,0) --(5,0);
\draw [thick](3,1) --(5,1);
\draw [thick] (3,0) --(3,1);
\draw [thick] (5,0) --(5,1);
\draw [fill] (3.5,0.2) circle [radius=0.03];
\draw [fill] (3.7,0.85) circle [radius=0.03];
\draw [fill] (4,0.1) circle [radius=0.03];
\draw [fill] (4.4,0.9) circle [radius=0.03];
\draw [thick] (3.5, 0.2)  -- (4,0.1) --(3.7,0.85) --(4.4,0.9);

\end{tikzpicture}
\caption{Changing a zigzag path into something more effective.}
\end{figure}
We introduce some notation. Let  $x_2, x_3, x_4$ be
i.i.d. variables distributed according to the exponential law $\sqrt{3}e^{-\sqrt{3}z}$
and let  $y_1, y_2, y_3, y_4$ be i.i.d. random variables uniformly distributed in $[0,\sqrt{3}]$.
We define four random points via
\begin{align*}
p_1 &=  (0, y_1) \qquad p_2  =(x_2, y_2) \qquad p_3 = (x_2+x_3, y_3) \qquad p_4 = (x_2 + x_3 + x_4, y_4).
\end{align*}
We know from the previous section that for all $1 \leq i \leq 3$
$$ \mathbb{E}\| p_i - p_{i+1}\|  =  \left(2\int_{0}^{\infty}{\int_{0}^{\sqrt{3}}{\sqrt{z_1^2+z_2^2}\exp{\left(-\sqrt{3}z_1\right)}
\left(1-\frac{z_2}{\sqrt{3}}\right)dz_2} dz_1}\right) \sim 0.92\dots$$

Given these four points, we introduce a stochastic event $(A)$.
\begin{align*}
\| p_1 - p_3\| + \|p_3 - p_2\| + \|p_2 - p_4\|  &\leq 
\| p_1 - p_2\| + \|p_2 - p_3\| + \|p_3 - p_4\| \qquad (A) 
\end{align*}
Furthermore, we will introduce the respective (random) difference
\begin{align*}
X &= (\| p_1 - p_2\| + \|p_2 - p_3\| + \|p_3 - p_4\|) - (\| p_1 - p_3\| + \|p_3 - p_2\| + \|p_2 - p_4\|)  
\end{align*}

\begin{lemma} We have
$$ \mathbb{E}\left(X \big| A \right) \mathbb{P}\left(A\right) \geq \frac{9}{4}10^{-6}.$$
\end{lemma}
\begin{proof} Since we are only trying to show a positive lower bound, rough estimates
suffice. We study the event $B$ defined as
$$ \left(x_2 \leq \frac{\sqrt{3}}{9}\right) \wedge \left(x_3 \leq  \frac{\sqrt{3}}{9}\right) \wedge 
\left(x_4  \leq  \frac{\sqrt{3}}{9}\right) \wedge \left(\min(y_1, y_3) \geq \frac{8\sqrt{3}}{9}\right)
\wedge \left(\max(y_2, y_4) \leq \frac{\sqrt{3}}{9}\right).$$
\begin{figure}
\begin{tikzpicture}[scale = 3]
\draw [thick] (0,0) --(0.7,0);
\draw [thick] (0,0) --(0,1);
\draw [thick] (0,1/9) --(0.7,1/9);
\draw [thick] (0,2/9) --(0.7,2/9);
\draw [thick] (0,3/9) --(0.7,3/9);
\draw [thick] (0,4/9) --(0.7,4/9);
\draw [thick] (0,5/9) --(0.7,5/9);
\draw [thick] (0,6/9) --(0.7,6/9);
\draw [thick] (0,7/9) --(0.7,7/9);
\draw [thick] (0,8/9) --(0.7,8/9);
\draw [thick] (0,9/9) --(0.7,9/9);
\draw [thick] (1/9,0) --(1/9,1);
\draw [thick] (2/9,0) --(2/9,1);
\draw [thick] (3/9,0) --(3/9,1);
\draw [thick] (4/9,0) --(4/9,1);
\draw [fill] (0.05,0.97) circle [radius=0.02];
\draw [fill] (0.16,0.05) circle [radius=0.02];
\draw [fill] (0.27,0.92) circle [radius=0.02];
\draw [fill] (0.37,0.06) circle [radius=0.02];
\end{tikzpicture}
\caption{An instance of the event $B$.}
\end{figure}
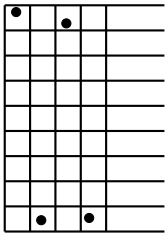
These variables are independent and all distributions are explicitely given: thus, for $2 \leq i \leq 4$, we have
$$\mathbb{P}\left(x_i \leq  \frac{\sqrt{3}}{9}\right) = \int_{0}^{\frac{\sqrt{3}}{9}}{\sqrt{3}e^{-\sqrt{3}z}dz} = 1-\frac{1}{e^{\frac{1}{3}}},$$
while $\mathbb{P}(\min(y_1, y_3) \geq \frac{8\sqrt{3}}{9}) = \mathbb{P} (\max(y_2, y_4) \leq \frac{\sqrt{3}}{9}) = 1/81$.
$$\mathbb{P}(B) = \left(1-\frac{1}{e^{\frac{1}{3}}}\right)^3\left(\frac{1}{81}\right)^2
\geq 3 \cdot 10^{-6}.$$
At the same time, a simple computation yields that in the event $B$, we always have
$$ X \geq \frac{3}{4}.$$
This implies that the event $B$ is a subset of the event $A$ and, trivially,
Therefore 
$$ \mathbb{E}\left(X \big| A \right) \mathbb{P}\left(A\right) \geq
 \mathbb{E}\left(X \big| B \right) \mathbb{P}\left(B\right)  \geq \frac{9}{4}10^{-6}.$$
\end{proof}

\begin{proof}[Proof of the upper bound.]  We follow the original idea of Beardwood, Halton \& Hammersley
and partition the unit square into strips: since we are dealing with a Poisson process, the behavior within
each strip is independent of that in all other strips; focusing on one strip, we are dealing with a Poisson
process of intensity $n$. For any set of random points arising from the Poisson process, we order them
 with increasing $x-$coordinate
$$\pi_1(x_1) \leq \pi_1(x_2) \leq \dots \leq \pi_1(x_k), $$
and consider the $4-$tuples $(x_1, x_2, x_3, x_4)$, $(x_5, x_6, x_7, x_8)$, and so on (with possibly
up to 3 points left at the end of each strip). Whether or not any of these $4-$tuples contains a zigzag 
structure is an independent event: the computations in the previous section then imply that with probability
at least $3 \cdot 10^{-6}$ a zigzag yielding a gain of at least $3/(4\sqrt{n})$ is present. There are
are $n/4 - O(\sqrt{n})$ $4-$tuples to consider implying the gain in length to be at of order
$(n/4-O(\sqrt{n}))(3/(4\sqrt{n}))(3 \cdot 10^{-6})$ and thus
$$ \beta \leq \beta_{BHH} -  \frac{9}{16}10^{-6}.$$
\end{proof}

\textbf{Remark.} The gain in length was achieved by looking at $n/4-O(\sqrt{n})$ independent events: usual
arguments would allow us to conclude that the predicted gain in length is actually tightly concentrated around
its mean. This, however, is not necessary for our sort of argument: we already know that $\beta$ describes the
limiting behavior almost surely: the expected length of any construction of deterministic paths is then necessarily
an upper bound on $\beta$.\\

\textbf{Remark.} These problems exhibit a concentration of measure phenomenon implying the stability of 
Monte-Carlo estimates, which will then usually imply much stronger results. For comparison, we did ten samples of a million random points each, which suggests
$$ \mathbb{P}\left(A\right) \sim 0.1418, \quad  \mathbb{E}\left(X \big| A \right) \sim 0.4187
\qquad \mbox{and thus} \qquad \mathbb{E}\left(X \big| A \right) \mathbb{P}\left(A\right) \geq 0.059.$$
with a standard deviation of 0.0003 and 0.001, respectively. This would imply that indeed
$$ \beta \leq 0.90632.$$

\subsection{Numerical estimates.} Our result was aimed towards the clearest presentation of the idea. Improvements
of the idea are rather obvious, however, they require somewhat accurate bounds for certain finite-dimensional integrals.
One particular generalization is as follows: one could study not merely zigzags but all 24 possible paths through six 
points leaving the first and the last point invariant; let us consider all 24 permutations over the symbols $\left\{2,3,4,5\right\}$
and denote the existence of an improved path as the stochastic event $(C)$
$$ \inf_{\pi \in S_4(\left\{2,3,4,5\right\})}{\|p_1 - p_{\pi(2)}\| + \sum_{i=2}^{4}{\|p_{\pi(i+1)} - p_{\pi(i)}\|}
+\|p_{6}-p_{\pi(5)}\|} < \sum_{i=1}^{5}{\|p_{i+1}-p_{i}\|} \qquad \qquad (C)$$
and the respective improvement by
$$ Z =  \sum_{i=1}^{5}{\|p_{i+1}-p_{i}\|} -  \inf_{\pi \in S_4(\left\{2,3,4,5\right\})}{\|p_1 - p_{\pi(2)}\| + \sum_{i=2}^{4}{\|p_{\pi(i+1)} - p_{\pi(i)}\|}}-\|p_{6}-p_{\pi(5)}\|.$$
Monte-Carlo methods (10 samples of 50000 sets of points each) suggest that
$$ \mathbb{P}\left(C\right) \sim 0.3721  \qquad \mbox{and} \qquad \mathbb{E}\left(Z \big| C \right) \sim 0.4990$$
with a standard deviation of 0.02 and 0.004, respectively. These values would suggest $\beta \leq 0.8902$. 
As already hinted at in the introduction, there is a natural limit to these improvements: we study paths through 
random points with an additional restriction on their movement in one of the two dimensions, which corresponds
to a different functional and this difference will be a great hindrance to further major improvements.

\section{Proof of the lower Bound}

\subsection{The original argument.} Proving an upper bound can (and has) been done by constructing an explicit path.
Proving a lower bound has to pursue an entirely different strategy since we have very little idea what an
optimal path could look like: we already know, however, that it is sufficient to prove lower bounds on the
expected length of the traveling salesman path through points of a Poisson process with intensity
$n$ in $[0,1]^2$. The only real basic information about paths at our disposal is that for every point there are two
points to which that particular point is connected: suppose now that for every point, these
two points are also the two closest points. \\

The second remark is that we may assume that the Poisson process is actually distributed with the
intensity $n$ on all of $\mathbb{R}^2$: adding more points can only decrease the expected distance
and allows us to disregard the behavior of the process close to the boundary of $[0,1]^2$. The following
 Lemma can be found in many basic books on probability theory.
\begin{lemma} Let $P_n$ be a Poisson process on $\mathbb{R}^2$ with intensity $n$. Then for
any fixed point $p \in \mathbb{R}^2$, the probability distribution of the distance between $p$ and the
nearest point in $\mathcal{P}_n$ is given by
$$ f(r) = 2\pi n r e^{-\pi n r^2}$$
and the distance to the second-nearest neighbour is given by
$$ g(r) = 2\pi^2 n^2 r^3 e^{-\pi n r^2}.$$
\end{lemma}
\begin{proof} We compute the probability of the closest point point lying at distance $(r, r+\varepsilon)$. This
is precisely the case if there is \textit{no} point in a disk of radius $r$ around $p$ but at least 1 point in
$(r, r+\varepsilon)$. It follows from the definition of the Poisson process that the probability of there being no point
in the disk is given by $e^{-n \pi r^2}$ and thus
$$ f(r) = \lim_{\varepsilon \rightarrow 0}{\frac{e^{-n \pi (r+\varepsilon)^2} - e^{-n \pi r^2}}{\varepsilon}} = 2r\pi n e^{-n \pi r^2}.$$
We compute the other expression in the same way: we require that there is precisely one point with
distance at most $r$ and another point at distance $(r, r+\varepsilon)$. The probability of being precisely
one point at distance at most $r$ is given by $ r^2\pi n e^{-r^2 \pi n}$ while the area of the annulus
is simply $((r+\varepsilon)^2-r^2)\pi$ and the probability of one point being in there is 
 $((r+\varepsilon)^2-r^2)\pi n e^{-((r+\varepsilon)^2-r^2)\pi n}.$ Altogether, we have
$$ g(r) =  (r^2\pi n e^{-r^2 \pi n}) \lim_{\varepsilon \rightarrow 0^+}{\frac{((r+\varepsilon)^2-r^2)\pi n e^{-((r+\varepsilon)^2-r^2)\pi n}}{\varepsilon}} = 2\pi^2 n^2 r^3 e^{-\pi n r^2}.$$

\end{proof}

Standard calculations 
give that the distance $r$ to the nearest point has expectation
$$ \int_0^{\infty}{rf(r)dr} = \frac{1}{2\sqrt{n}}$$
while the distance to the next-to-nearest point has expectation
$$ \int_0^{\infty}{rg(r)dr} = \frac{3}{4\sqrt{n}}.$$
Given a traveling salesman path, every point is connected to two other points -- in the worst case, these are the nearest and the next-to-nearest point in all cases, yielding a lower bound of
$$ \beta \geq \left(\frac{1}{2}+\frac{3}{4}\right)\frac{1}{2} = 0.625,$$
which is the original result of Beardwood, Halton \& Hammersley.

\subsection{An improvement.} The previous argument assumed that it is always the
worst case that occurs: every point is connected to its two closest neighbors. This, however,
is not possible if we have the following constellation of points: a point $a$ with closest 
point $b$ at distance $r_1$ and its second-closest point $c$ at distance $r_2 > r_1$ 
and third-closest point $d$ at distance $r_3 > r_1 + 2r_2$. Our proof rests on an
analysis of this situation.
\begin{lemma} Let $P_n$ be a Poisson process on $\mathbb{R}^2$ with intensity $n$. Then for
any fixed point $p \in \mathbb{R}^2$, the probability distribution of the distance between $p$ and the
closest, second closest and third closest point is given by
$$ h(r_1, r_2, r_3) = \begin{cases}
e^{-n \pi r_3^2}(2n\pi)^3r_1 r_2 r_3 \qquad &\mbox{if~}r_1 < 
r_2 < r_3 \\
0 \qquad &\mbox{otherwise.}
\end{cases}$$
\end{lemma}
\begin{proof} As before, we study the probability of precisely one point at distance $(r_1, r_1+\varepsilon)$ (event A), precisely one point in $(r_2, r_2+\varepsilon)$ (event B) and precisely one point in $(r_3, r_3+\varepsilon)$ (event C) and no points in between (event D). The probabilities for these events 
including their expansion up to first order in $\varepsilon$ are
\begin{align*}
\mathbb{P}(A) &= ((r_1+\varepsilon)^2-r_1^2)n\pi e^{(-(r_1+\varepsilon)^2+r_1^2)n\pi}
=  2n\pi\varepsilon r_1 + O(\varepsilon^2)\\
\mathbb{P}(B) &= ((r_2+\varepsilon)^2-r_2^2)n\pi e^{(-(r_2+\varepsilon)^2+r_2^2)n\pi} 
=  2n\pi\varepsilon r_2 + O(\varepsilon^2)\\
\mathbb{P}(C) &= ((r_3+\varepsilon)^2-r_3^2)n\pi e^{(-(r_3+\varepsilon)^2+r_3^2)n\pi} 
=  2n\pi\varepsilon r_3 + O(\varepsilon^2)\\
\mathbb{P}(D) &= e^{-r_1^2\pi n}e^{(-r_2^2+(r_1+\varepsilon)^2)\pi n}e^{(-r_3^2+(r_2+\varepsilon)^2)\pi n} = e^{-n\pi r_3^2} + P(\varepsilon).
\end{align*}
This immediately implies the statement.
\end{proof}

\begin{proof}[Proof of the lower bound.] We start by showing that both the nearest as well as the next-to-nearest
point of any element in $\left\{a,b,c\right\}$ also lies in the set. Let $x$ be some other point with $x \notin \left\{a,b,c\right\}$. Then
$$ \| b - x\| \geq \|a - x\| - r_1 \geq r_3 - r_1 > 2r_2 \geq r_1 + r_2 \geq \|b-c\|$$
and therefore the second closest point from $b$ is $a$ or $c$. By the same token
$$ \| c - x\| \geq \|a - x\| - r_2 \geq r_3 - r_2 \geq r_1 + r_2 \geq \|b-c\|$$
and therefore the second closest point from $c$ is $a$ or $b$. Cases of equality have
probability 0 and can be ignored.

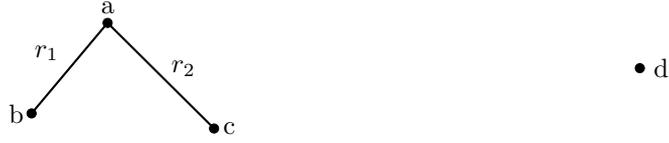
\begin{figure}[h!]
\begin{tikzpicture}[xscale = 2, yscale = 2]
\draw [fill] (0.5,0.2) circle [radius=0.03];
\draw [fill] (1,0.8) circle [radius=0.03];
\draw [fill] (1.7,0.1) circle [radius=0.03];
\draw [fill] (4.5,0.5) circle [radius=0.03];
\draw [thick] (0.5,0.2) --(1,0.8) -- (1.7,0.1);

\draw (1,0.9) node{a};
\draw (0.4,0.2) node{b};
\draw (1.8,0.1) node{c};
\draw (0.6,0.6) node{$r_1$};
\draw (1.5,0.5) node{$r_2$};
\draw (4.64,0.5) node{d};
\end{tikzpicture}
\caption{$a$ and the three closest points of $a$.}
\end{figure}
If we simply connect every point to its two closest neighbours, we end up with a triangle where
every point is connected to the two other points but no other point except those. This is clearly
not possible for a traveling salesman path. Let us first compute the frequency of such an event.
Using the Lemma, the probability of all of these distance relations being true for a fixed point $a$ is 
$$ \int_{0}^{\infty}{\int_{r_1}^{\infty}{\int_{r_1+2r_2}^{\infty}{e^{-n \pi r_3^2}(2n\pi)^3r_1 r_2 r_3
dr_3}dr_2}dr_1} = \frac{7}{324}.$$
There is a lack of independence: if it is true for $a$, it is likely to be true for
$b$ and $c$ as well -- thus, we have only the trivial bound 
$$ \frac{1}{3}\frac{7}{324}n = \frac{7n}{972}$$
on the number of triples of points with this property. However, if the case occurs, then the
algorithm connecting every point to its two nearest neighbours has an expected
length which can be bounded from above by
$$ r_1 + r_2 +2||a-c|| \leq 3(r_1 + r_2).$$
where the distance $\|a-c\|$ has to be counted twice because the algorithm cannot
'see' that it has created a triangle and counts the distance twice.
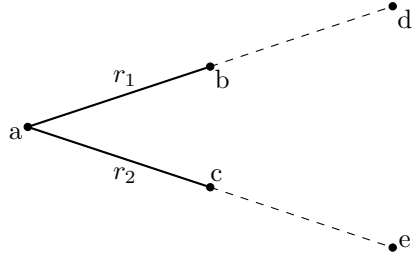
\begin{figure}[h!]
\begin{tikzpicture}[xscale = 0.8, yscale = 0.8]
\draw [fill] (0,0) circle [radius=0.06];
\draw (-0.2,-0.1) node{a};

\draw [fill] (3,1) circle [radius=0.06];
\draw (3.2,0.8) node{b};
\draw [thick] (0,0) -- (3,1);
\draw (1.6,0.8) node{$r_1$};

\draw [fill] (3,-1) circle [radius=0.06];
\draw (3.1,-0.8) node{c};
\draw [thick] (0,0) -- (3,-1);
\draw (1.6,-0.8) node{$r_2$};

\draw [fill] (6,2) circle [radius=0.06];
\draw (6.2,1.8) node{d};
\draw [dashed] (3,1) -- (6,2);

\draw [fill] (6,-2) circle [radius=0.06];
\draw (6.2,-1.9) node{e};
\draw [dashed] (3,-1) -- (6,-2);
\end{tikzpicture}
\caption{The best of the worst case.}
\end{figure}
In the case of three points isolated from the rest, there is one special case which is the
easiest to connect to the remaining points: this is is when $b$ can be connected to a
point $d$ having distance $r_3$ from $a$ and $c$ can be connected to a different point $e$ also
at distance $r_3$ from $a$ and, additionally, $b$ lies on the line $\overline{ad}$ and $c$ 
lies on $\overline{ae}$. In this case, the
required length is 
$$\|d-b\| + \|b-a\| + \|a-c\| + \|c-e\| \geq (r_3-r_1) + r_1 + r_2 + (r_3-r_2) = 2r_3.$$
This implies that whenever we are in this particular configuration, the actual path has 
to be at least a length $2r_3 - 3(r_1 + r_2)$ longer than what the greedy algorithm suggests. 
Note that
$$ 2r_3 - 3(r_1 + r_2) \geq r_2-r_1 \geq 0$$
and that we always gain something at this point.
In expectation, this is an average length of
$$ \frac{324}{7}\int_{0}^{\infty}{\int_{r_1}^{\infty}{\int_{r_1+2r_2}^{\infty}{(2r_3-3r_1-3r_2)e^{-n \pi r_3^2}(2n\pi)^3r_1 r_2 r_3
dr_3}dr_2}dr_1} = \frac{57}{112}\frac{1}{\sqrt{n}}.$$
Altogether, this gives the lower bound
$$ \beta \geq \left(\frac{1}{2}+\frac{3}{4}\right)\frac{1}{2} + \frac{7}{972}\frac{57}{112} = \frac{5}{8} + \frac{19}{5184}.$$
\end{proof}

\textbf{Acknowledgments.} I am grateful to Steven Finch for comments on the history
of the problem and indebted to J. Michael Steele for encouraging me to write this paper
and a series of very valuable discussions. The author was supported by SFB 1060 of the DFG.

\end{document}